\documentclass[reqno]{amsart}

\usepackage{enumitem}
\usepackage{hyperref}
\usepackage{float}
\usepackage[all]{xy}
\usepackage{graphicx}
\usepackage{subfigure}
\usepackage{amsthm, amssymb, amsmath,mathrsfs}
\usepackage[utf8]{inputenc}

\newtheorem{Th}{Theorem}[section]
\newtheorem{Cor}{Corollary}[section]
\newtheorem{Lemma}{Lemma}[section]
\newtheorem{Def}{Definition}[section]
\newtheorem{Rem}{Remark}[section]

\newcommand{\E}{\mathrm{e}}
\newcommand{\I}{\mathrm{i}}

\newcommand*{\mailto}[1]{\href{mailto:#1}{\nolinkurl{#1}}}
\newcommand{\arxiv}[1]{\href{http://arxiv.org/abs/#1}{arXiv:#1}}

\numberwithin{equation}{section}

\begin{document}
\title[Uncertainty Princ. for disc. Schr\"odinger evolution on graphs]{Uncertainty principle for   discrete Schr\"odinger evolution on  graphs}
\author[I. Alvarez-Romero]{Isaac Alvarez-Romero}
\address{Department of Mathematical Sciences,
Norwegian University of
Science and Technology, NO--7491 Trondheim, Norway}
\email{\mailto{isaac.romero@math.ntnu.no}\\
\mailto{isaacalrom@gmail.com}}

\thanks{{\it Research supported by the Norwegian Research Council project DIMMA 213638.}}


\begin{abstract}
We consider the Schr\"odinger evolution on graph, i.e. solution to the equation $\partial_tu(t,\alpha)=i\sum_{\beta\in\mathcal{A}}L(\alpha,\beta)u(t,\beta)$, here $\mathcal{A}$ is the set of vertices of the graph and the matrix $(L(\alpha,\beta))_{\alpha,\beta\in\mathcal{A}}$ describes interaction between the vertices, in particular two vertices $\alpha$ and $\beta$ are connected if $L(\alpha,\beta)\neq0$. We assume that the graph has a "web-like" structure, i.e, it  consists of an inner part, formed by a finite number of vertices, and some threads attach to it.\\
We prove that such solution $u(t,\alpha)$ cannot decay too fast along one thread
at two different times, unless it vanishes at this thread.\\
We also  give a characterization of the dimension of the vector space formed by all the solutions of  $\partial_tu(t,\alpha)=i\sum_{\beta\in\mathcal{A}}L(\alpha,\beta)u(t,\beta)$ when $\mathcal{A}$ is a finite set, in terms of the number of the different eigenvalues of the matrix $L(\cdot,\cdot)$
\end{abstract}

\maketitle

\section{Introduction}

The Hardy Uncertainty Principle has been studied by several authors in the continuos case in  recent years, see for instance \cite{CEKPV,EKPV} and the references therein. This Principle can be formulated in terms of the dynamic version for the free  Schr\"odinger equation: let  $u(t,x)$ be a solution of
\begin{equation}
\partial_t u=\I \Delta u
\end{equation}
and $|u(0,x)|=O(\E^{-x^2/\beta^2})$, $|u(1,x)|=O(\E^{-x^2/\alpha^2})$, with $1/\alpha\beta>1/4$, then $u\equiv 0$ and if $1/\alpha\beta=1/4$, then the initial data is a constant multiple of $\E^{-(1/\beta^2+\I/4)x^2}$. A similar result is given in \cite{JLMP,FB} for the discrete case, that is when $\Delta u(t,n)=u(t,n-1)-2u(t,n)+u(t,n+1)$ is the discrete Laplacian and $n\in\mathbb{Z}$, $t\in[0,1]$. Such  type of operators appears , for instance, on the study of the quantum graphs, see for example \cite{BK,K} and the references therein. \\ \\
The aim of the present paper is to study the uniqueness of solutions for the discrete Schr\"odinger evolution on connected graphs. We suppose that the graphs have a  "web-like" structure, that is, there exists a central part $\mathcal{A}_1$, which consits of a finite number of vertices, and some threads attached to $\mathcal{A}_1$.  We denote by $\mathcal{A}$ the set of all vertices, a detailed description is given  in section \ref{section22}.\\
These systems appear, for example, when one considers a system of particles interacting with each other and perhaps an external field, see \cite{LM,LM2}. These interactions are described by the matrix
\begin{equation*}
\big(L(\alpha,\beta)\big)_{\alpha,\beta\in\mathcal{A}}
\end{equation*}
This matrix is symmetric and real-valued. The operator $\big(L(\alpha,\beta)\big)_{\alpha,\beta\in\mathcal{A}}:\l^2(\mathcal{A})\to\l^2(\mathcal{A})$ is related to the Hessian matrix of the potential energy function near the equilibrium position of the particles, thus $L(\cdot,\cdot)$ is a positive and selfadjoint operator.\\
There is a  graph  that describes these systems:  the vertices play the role of the particles and the edges describe the interactions, that is, there is an edge $(\alpha,\beta)$ if $\alpha\neq\beta$ and  the particle $\alpha$ interacts with $\beta$,  i.e.  $L(\alpha,\beta)\neq 0$.  
The evolutions on such graphs are described by functions $u(t,\alpha)$, $t\in[0,1]$, $\alpha\in\mathcal{A}$ and they meet the equation
\begin{equation}\label{introproblem}
\partial_tu(t,\alpha)=\I\sum_{\beta\in\mathcal{A}}L(\alpha,\beta)u(t,\beta),\qquad \alpha\in\mathcal{A},\qquad t\in[0,1]
\end{equation}
We will show that if a solution of \eqref{introproblem} decays sufficiently fast on one thread at two different times, then the solution is trivial on the whole thread. To this end we will combine techniques on scattering theory on such graphs, developed in \cite{LM,MS}  and techniques of  the growth of entire functions,  present e.g. in  \cite{Levin}, to follow a similar strategy as it was done in \cite{JLMP} in theorem 2.3, where it was  proven  that if a solution $u(t,n)$ of the problem
\begin{equation}\label{introproblemJLMP}
\partial_tu(t,n)=\I(\Delta u(t,n) +V(n)u(t,n)),\qquad n\in\mathbb{Z},\qquad t\in[0,1]
\end{equation}
decays sufficiently fast on one side at two different times, then the solution is trivial, where $\Delta u(n)$ is the discrete Laplacian and $V(n)$ is a compactly supported real valued function.\\
In \cite{ART}, this result was improved by letting  $L$ be a Jacobi operator: 
\begin{equation*}
Lf(n)=-b(n-1)f(n-1)+a(n)f(n)-b(n)f(n+1),\qquad n\in\mathbb{Z}
\end{equation*}
such that the sequences $a,b$ fullfill  certain decay conditions as $n\to\pm\infty$.\\  \\
When we have several threads, in general we cannot assure that the function $u(t,\alpha)$ is trivial in the whole system, once we know it is zero on one thread. This simple but very important fact is a big difference with \eqref{introproblemJLMP}.  One of the reasons of this issue is because of the inner part $\mathcal{A}_1$. Motivated by this fact, we will restrict \eqref{introproblem}  to the case when $\mathcal{A}$ is a finite set and  show some cases  when it is possible to extend the solution to the whole system: \\
\begin{enumerate}
\item
Either if we  know the behaviour of the solution $u(t,\alpha)$ along the threads. In fact, if we know the solution on all the threads, except one, and we can extend the graph formed by these threads to the whole system, then we know the solution on the whole system, a detailed description is given in corollary \ref{cor43}.\\
\item
Either if we know the solution on the central part $\mathcal{A}_1$ and there is an extension of $\mathcal{A}_1$ to the whole system, then $u(t,\alpha)$ is uniquely determined for all $\alpha\in\mathcal{A}$, see corollary \ref{cor42}
\end{enumerate}
These ideas are based on \cite{MS}.\\ \\
The paper is organized as follows: in section \ref{Preli} we give some brief notions of the growth of entire functions, see \cite{Levin}, and some results of  the scattering problem on the considered graphs, see \cite{LM}. We need them to prove in section \ref{sec3} our result on the uniqueness of the solution of \eqref{introproblem} on one thread, that is, when  the solution decays sufficiently fast at two different times on that thread.\\
In section \ref{section4} we study the problem \eqref{introproblem} restricted to finite graphs and we give a complete characterization of the dimension of the vector space formed by all the solutions $u(t,\alpha)$ in terms of the number of the different eigenvalues of the matrix $L(\cdot,\cdot)$. To this end, we need the concept of the extension of a subgraph developed in \cite{LM2} and \cite{MS},  in chapter 12.

\section{Preliminaries}\label{Preli}

\subsection{Growth of entire functions}
We will give some brief notions of the growth of the entire functions and some results related with it, all of them can be found in \cite{Levin}  in lectures 1 and 8.\\ \\
Let $f$ be an entire function, we say that $f$ is of \emph{exponential type } $\sigma_f$, if for  some constants $k,C>0$  we have
\begin{equation}\label{2fdecay}
|f(z)|<C\E^{k|z|},\qquad z\in\mathbb{C}
\end{equation}
and $\sigma_f$ is defined as
\begin{equation*}
\sigma_f=\limsup_{r\to\infty}\frac{\log\max\{|f(r\E^{i\varphi})|:\text{ }\varphi\in[0,2\pi]\}}{r}
\end{equation*}
It follows from the definition of $\sigma_f$ 
\begin{equation}
\sigma_{fg}\leq \sigma_f+\sigma_g
\end{equation}
and
\begin{equation}\label{sumasigma}
\sigma_{f+g}\leq\max\{\sigma_f,\sigma_g\}
\end{equation}
where $f,g$ are entire functions of exponential type $\sigma_f,\sigma_g$ respectively.
\begin{Th}\label{Thtypeformula}
Let $f=\sum_{n\geq 0}c_nz^n$ be an entire function of exponential type, then $\sigma_f$ is determined by the formula
\begin{equation}
\E\sigma_f=\limsup_{n\to\infty}\big(n|c_n|^{1/n}\big)
\end{equation}
\end{Th}
Let $f$ be an entire function, it may happen that it does not growth with the same velocity along all directions. To this end we introduce the \emph{indicator function} $h_f$
\begin{equation*}
h_f(\varphi)=\limsup\frac{\log|f(r\E^{i\varphi})|}{r}
\end{equation*}
where $\varphi\in[0,2\pi]$ is the direction we are looking at, that is $\text{arg}(z)=\varphi$.\\
\begin{Def}
A function $K(\theta)$ is called trigonometrically convex on the closed segment $[\alpha,\beta]$ if for $\alpha\leq \theta_1<\theta_2\leq\beta$, $0<\theta_2-\theta_1<\pi$ we have
\begin{equation*}
K(\theta)\leq\frac{K(\theta_1)\sin(\theta_2-\theta)+K(\theta_2)\sin(\theta-\theta_1)}{\sin(\theta_2-\theta_1)},\qquad \theta_1\leq\theta\leq\theta_2
\end{equation*}
\end{Def}
\begin{Lemma}
Let $h(\varphi)$ be a trigonemetrically convex function on the segment $[\alpha,\beta]$. Then
\begin{equation*}
h(\varphi)+h(\varphi+\pi)\geq 0,\qquad \alpha\leq\varphi<\varphi+\pi\leq\beta
\end{equation*}
\end{Lemma}
\begin{Th}\label{thindicator1}
Let $f(z)$ be an entire function of exponential type. Then its indicator function $h_f$ is a trigonometrically convex function.
\end{Th}
As a consequence we note
\begin{Cor}
Let $f(z)$ be an entire function of exponential type, then
\begin{equation}\label{hineq}
h_f(\varphi)+h_f(\varphi+\pi)\geq 0
\end{equation}
\end{Cor}
\begin{Rem}
\begin{enumerate}
\item
Notice that if $f(z)$ is a function such that  fullfills \eqref{2fdecay} and has  a finite number of poles, then  there exists a polinomial $P(z)$ such that $\tilde{f}(z)=f(z)P(z)$ is an entire function and
\begin{equation*}
\sigma_{\tilde{f}}=\sigma_f,\qquad h_{\tilde{f}}=h_f
\end{equation*}
\item
In addition if $f(z)$ is an analitic function on $\mathbb{C}\setminus\overline{\mathbb{D}_r}$, here $\overline{\mathbb{D}_r}:=\{z\in\mathbb{C}:\text{ }|z|\leq r\}$, then all the definitions upper mentioned can be extended to this kind of functions. Moreover,  since the key part of  the proof of theorem \ref{thindicator1} is the Phragm\'en-Lindel\"of theorem, one can easily adapt the proof of Theorem 1 from Chapter in \cite{Levin} to show that it continuos to hold in this case. In particular, inequality \eqref{hineq} is still true in this case
\end{enumerate}
\end{Rem}

\subsection{Direct Multichannel Scattering Problem}\label{section22}
The detailed description of this problem is given in \cite{LM}. In this subsection we repeat it in order to introduce notation and also to make our exposition self-contained.\\ \\
Consider a set of particles $\mathcal{A}$ and study small oscillations around its equilibrium position. The particles interacts each other and possibily with an external field. This problem is reduced to  the spectral problem
\begin{equation}\label{SpectralL}
Lx=\lambda x
\end{equation}
Here $L:\l^2(\mathcal{A})\to\l^2(\mathcal{A})$ is a selfadjoint operator, symmetric,  real-valued and positive with matrix $(L(\alpha,\beta))_{\alpha,\beta\in\mathcal{A}}$. In the sequel we will not distinguish between the operator $L$ and its corresponding matrix $L(\cdot,\cdot)$.\\ \\
Let $\alpha,\beta$ be two particles, we say that they interact if $L(\alpha,\beta)\neq 0$.  The particles are distributed  in a finite set $\mathcal{A}_1$ and a finite set of channels which are attached to this set $\mathcal{A}_1$
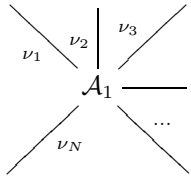
\begin{figure}[H]
\begin{displaymath}
\xymatrix{& & \\ & \mathcal{A}_1 \ar@{-}[ul]^{\nu_1}\ar@{-}[u]^{\nu_2}\ar@{-}[ur]^{\nu_3}\ar@{-}[dl]^{\nu_N}\ar@{-}[r]\ar@{-}[dr]^{\dots}& \\& & }
\end{displaymath}\caption{General picture of the system}\label{Msystem}
\end{figure}
where $\nu_j$ denotes a set of infinitely many  particles where each element $\nu_j(k)\in\nu_j$  interacts with two more different from itself and no other particle outside  $\nu_j$, except  the ending point ($\nu_j(0)$), that is 
\begin{equation*}
\nu_j=\{\nu_j(k)\}_{k\geq 0}, \quad L(\nu_j(k),\nu_j(n))=0 \text{ if }|n-k|>1,\qquad 1\leq j\leq N
\end{equation*}
These sets are called \emph{channels} and we denote $\mathcal{C}=\cup_{j=1}^N\{\nu_j\}$ and 
\begin{equation*}\mathcal{A}_0=\cup_{\nu\in\mathcal{C}}\cup_{k\geq 1}\{\nu(k)\}
\end{equation*}
\begin{figure}[H]
\begin{displaymath}
\xymatrix{(\mathcal{A}_1\setminus{\nu(0)})\ar@{-}[r]&\nu(0)\ar@{-}[r]&\nu(1)\ar@{-}[r]&\dots\ar@{-}[r]&\nu(k)\ar@{-}[r]&\dots\ar@{-}[r]&}
\end{displaymath}\caption{Representation of an arbitrary channel $\nu\in\mathcal{C}$}
\end{figure}
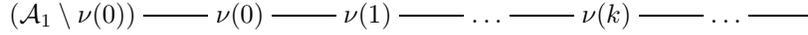
And  $\mathcal{A}_1$ in figure \ref{Msystem} is defined by $\mathcal{A}_1=\mathcal{A}\setminus\mathcal{A}_0$.\\ \\
Equation \eqref{SpectralL} is written now as
\begin{equation}\label{SpectralAL}
\lambda x(\alpha)=\sum_{\beta\in\mathcal{A}}L(\alpha,\beta)x(\beta),\qquad \alpha\in\mathcal{A}
\end{equation}
To simplify notation, for a channel $\nu\in\mathcal{C}$ and $k\geq 1$, we set
\begin{equation}\label{abnu}
-b_\nu(k-1)=L(\nu(k-1),\nu(k)),\quad a_\nu(k)=L(\nu(k),\nu(k)),\quad \nu(k)\in\nu
\end{equation}
In the sequel  we assume that the sequences $b_\nu,a_\nu$ are stabilized after some $K_0>0$ for all $\nu\in\mathcal{C},$ that is
\begin{equation*}
\begin{cases}& b_\nu(k)=1\quad \text{ if }k\geq K_0\\
&a_\nu(k)=2\quad \text{ if }k\geq K_0\end{cases}
\end{equation*}
Let $\alpha\in\mathcal{A}_1$, then \eqref{SpectralAL} can be expressed as
\begin{equation}\label{SpectralA1}
\lambda x(\alpha)-\sum_{\beta\in\mathcal{A}_1}L(\alpha,\beta)x(\beta)=\sum_{\beta\in\mathcal{A}_0}L(\alpha,\beta)x(\beta)=\sum_{\nu\in\mathcal{C}}L(\alpha,\nu(1))x(\nu(1))
\end{equation}
The last equality follows from the fact that the only pairs $(\alpha,\beta)\in\mathcal{A}_1\times\mathcal{A}_0$ such that $L(\alpha,\beta)\neq 0$ are of the form $(\nu(0),\nu(1))$, $\nu\in\mathcal{C}$.\\
Let $L_1=(L(\alpha,\beta))_{\alpha,\beta\in\mathcal{A}_1}$ be a submatrix of the operator $L$. It follows that $L_1$ is real, symmetric and positive. Let $0<\lambda_1\leq\dots\leq\lambda_M$ and $p_1,\dots,p_M\in l^2(\mathcal{A}_1)$ be its eigenvalues and the respective normalized eigenvectors, which can be chosen to be real-valued.  Here $M=\sharp\mathcal{A}_1$  and for $\lambda\notin\{\lambda_j\}_{j=1}^M$, the operator $L_1-\lambda I$ is invertible and \eqref{SpectralA1} turns into
\begin{equation*}
\begin{split}
x(\alpha)=&-\sum_{\beta\in\mathcal{A}_1}r(\alpha,\beta;\lambda)\sum_{\nu\in\mathcal{C}}L(\beta,\nu(1))x(\nu(1))=\\
&\sum_{\nu\in\mathcal{C}}r(\alpha,\nu(0);\lambda)b_\nu(0)x(\nu(1)),\quad \alpha\in\mathcal{A}_1
\end{split}
\end{equation*}
where $r(\alpha,\beta;\lambda)$ are the entries of the resolvent $\mathscr{R}=(L_1-\lambda I)^{-1}$ of the matrix $L_1$:
\begin{equation}\label{Rmatrix}
\mathscr{R}=(r(\alpha,\beta;\lambda))_{\alpha,\beta\in\mathcal{A}_1},\quad r(\alpha,\beta;\lambda)=\sum_{l=1}^M\frac{p_l(\alpha)p_l(\beta)}{\lambda_l-\lambda}
\end{equation}
Thus, for $\alpha=\nu(0)$ we obtain
\begin{equation}\label{bdrycondition}
x(\nu(0))=\sum_{\sigma\in\mathcal{C}}r(\nu(0),\sigma(0);\lambda)b_\sigma(0)x(\sigma(1)),\qquad \nu\in\mathcal{C}
\end{equation}
This relation links the values of a solution $x(\alpha)$  on $\mathcal{A}_0$ and on $\cup_{\nu\in\mathcal{C}}\{\nu(0)\}\subset\mathcal{A}_1$. We refer to it as the \emph{boundary condition}.\\ \\
Consider now the finite-difference equation, see for instance \cite{Tes} and \cite{MS}  in particular chapters  1 and 2 :
\begin{equation}\label{eqchannel}
-b(k-1)x(k-1)+a(k)x(k)-b(k)x(k+1)=\lambda x(k),\qquad k=1,2,\dots
\end{equation}
where the sequences $a, b$ are real valued and for  every $k>K_0>0$, $a(k)=2$ and $b(k)=1$. Set $\lambda=\lambda(\theta)$ as follows
\begin{equation}\label{lambdapara}
\begin{split}
\lambda:\quad&\mathbb{T}\longrightarrow[0,4]\\
&\theta\longrightarrow\lambda(\theta):=2-\theta-\theta^{-1}
\end{split}
\end{equation}
Then there exists linear independent solutions of \eqref{eqchannel}: $e(k,\theta),e(k,\theta^{-1})$, such that $\overline{e(k,\theta)}=e(k,\theta^{-1})$ and
\begin{equation}\label{Jostsolution}
e(k,\theta)=\begin{cases}\theta^k\quad&\text{ if }k>K_0\\
\sum_{n\geq k}^{K_0+1}c(n,k)\theta^n\quad&\text{ otherwise }\end{cases}
\end{equation}
here $\theta\in\mathbb{T}\setminus\{\pm1\}$ and $c(n,k)$ are constants. Thus every solution $\xi(k,\theta)$ of \eqref{eqchannel} can be expressed as
\begin{equation}\label{mntheta}
\xi(k,\theta)=m(\theta)e(k,\theta)+n(\theta)e(k,\theta^{-1}),\qquad k=0,1,2,\dots
\end{equation}
Notice that $m(\theta), n(\theta)$ are constants that depends on $\theta$ only. \\ \\
Consider the sequences $a_\nu,b_\nu$ defined in \eqref{abnu} and the corresponding finite-difference equations \eqref{eqchannel} for each channel $\nu\in\mathcal{C}$. Set the matrices
\begin{equation*}
\begin{split}
&E(k,\theta):=\text{diag}\{e_\nu(k,\theta)\}_{\nu\in\mathcal{C}},\qquad B(k)=\text{diag}\{b_\nu(k)\}_{\nu\in\mathcal{C}}\\
&R(\theta)=(r(\nu(0),\sigma(0);\lambda(\theta))_{\nu,\sigma\in\mathcal{C}}:\text{ }l^2(\mathcal{C})\to l^2(\mathcal{C})\\
&R_1(\theta)=(r(\alpha,\nu(0);\lambda(\theta))_{\alpha\in\mathcal{A}_1,\nu\in\mathcal{C}}:\text{ }l^2(\mathcal{C})\to l^2(\mathcal{A}_1)
\end{split}
\end{equation*}
where the functions $r(\alpha,\beta;\lambda)$ are defined in \eqref{Rmatrix} and $e_\nu(k,\theta)$ are the corresponding solutions of \eqref{eqchannel} for $\nu\in\mathcal{C}$ defined in \eqref{Jostsolution}. Let $\textbf{m}=\{m_\nu(\theta)\}_{\nu\in\mathcal{C}}$ and $\textbf{n}=\{n_\nu(\theta)\}_{\nu\in\mathcal{C}}$, here $m_\nu(\theta),n_\nu(\theta)$ are the constants in \eqref{mntheta} corresponding to the channel $\nu$. Then the boundary condition \eqref{bdrycondition} acquires the form
\begin{equation*}
E(0,\theta)\textbf{m}+E(0,\theta^{-1})\textbf{n}=R(\theta)B(0)(E(1,\theta)\textbf{m}+E(1,\theta^{-1})\textbf{n})
\end{equation*}
or 
\begin{equation}\label{TmT1n}
T(\theta)\textbf{m}=-T(\theta^{-1})\textbf{n}
\end{equation}
where
\begin{equation}\label{Tmatrix}
T(\theta)=E(0,\theta)-R(\theta)B(0)E(1,\theta)
\end{equation}
The matrices  $R(\theta)$ and $T(\theta)$ are well defined for all $\theta\in\bar{\mathbb{D}}\setminus\mathscr{O}$, where
\begin{equation}\label{Odef}
\mathscr{O}=\{\theta\in\bar{\mathbb{D}}:\text{ }\lambda_l-\lambda(\theta)=0\text{ for some }l\in\{1,\dots,M\}\}\cup\{-1,0,1\}
\end{equation}
\begin{Lemma}
The inequality
\begin{equation*}
|\langle\overline{E(1,\theta)}B(0)T(\theta)x,x\rangle|\geq|\text{Im}\langle\overline{E(1,\theta)}B(0)T(\theta)x,x\rangle|\geq\frac{|\bar{\theta}-\theta|}{2}\Delta||x||^2
\end{equation*}
holds for all $\theta\in\bar{\mathbb{D}}\setminus\mathscr{O}$ and $x\in l^2(\mathcal{C})$.
\end{Lemma}
As a consequence we have
\begin{Cor}
The operators $T(\theta)$ are invertible for all nonreal $\theta\in\bar{\mathbb{D}}\setminus\mathscr{O}$
\end{Cor}
Thus, for $\theta\in\bar{\mathbb{D}}\setminus\mathscr{O}$, equation \eqref{TmT1n} implies that
\begin{equation*}
\textbf{m}=S(\theta)\textbf{n}
\end{equation*}
where
\begin{equation*}
S(\theta)=-T(\theta)^{-1}T(\theta^{-1})=(s(\sigma,\nu;\theta))_{\sigma,\nu\in\mathcal{C}}
\end{equation*}
Thus, all solutions $\varphi(\alpha,\theta)$ of \eqref{SpectralL} are of the form
\begin{equation}\label{varphialpha}
\varphi(\alpha,\theta)=\begin{cases}\Big(U(k,\theta)\textbf{n}\Big)(\nu(k))&\text{ if }\alpha=\nu(k),\text{ for some }\nu\in\mathcal{C},\text{ } k\geq 0\\
\Big(R_1(\theta)B(0)U(1,\theta)\textbf{n}\Big)(\alpha)&\text{ if }\alpha\in\mathcal{A}_1\end{cases}
\end{equation}
with an arbitrary $\textbf{n}\in \l^2(\mathcal{C})$ and
\begin{equation*}
U(k,\theta)=E(k,\theta^{-1})+E(k,\theta)S(\theta),\qquad\theta\in\mathbb{T}\setminus\mathscr{O}, \quad k=0,1,\dots
\end{equation*}
Actually the function $U(k,\theta)$ can be extended to a meromorphic function inside the unit disk  and 
\begin{Lemma}\label{lemmaUpoles}
There is a finite set $\Theta\subset\mathbb{D}$ such that for all $k\geq0$ the poles of $U(k,\theta)$ are located in $\Theta\cup\{0\}$. In addition the order of the pole in the origin is $k$ and the rest of poles are simple. The matrix functions $U(k,\theta)$ are bounded in sufficiently small annulus $1-\epsilon\leq |\theta|\leq 1$.
\end{Lemma}
The exact statement of the lemma \ref{lemmaUpoles} is shown in \cite{LM} in  section 4,  in Lemma 4.1, Lemma 4.2 and their Corollary.
\section{Discrete Schr\"odinger evolution}\label{sec3}
Before we establish our main result of this section (theorem \ref{thdecayuniq}), we need the following technical lemma
\begin{Lemma}\label{Lemmarational}
The entries of the matrices $U(k,\theta)$, $k\geq 0$, and $R_1(\theta)B(0)U(1,\theta)$ in\eqref{varphialpha} are rational functions with respect to $\theta$
\end{Lemma}
\begin{proof}
It follows from \eqref{Rmatrix} that the entries of $r(\alpha,\beta;\lambda)$ of the resolvent $\mathscr{R}$ are rational functions with respect to $\lambda$. Using \eqref{lambdapara}, $\lambda(\theta)$ is a rational function with respect to $\theta$, whence $r(\alpha,\beta;\lambda)$ are rational functions with respect to $\theta$ too. Since the entries of $B(0)$ are constants, it remains to show that $U(k,\theta)$, $k\geq 0$, are also rational, but this follows from \eqref{Jostsolution} and \eqref{Tmatrix}. 
\end{proof}
\begin{Th}\label{thdecayuniq}
Let $u(t,\alpha)\in C^1([0,1],\l^2(\mathcal{A}))$ be a solution of
\begin{equation}\label{evouproblem}
\partial_t u(t,\alpha)=\I\sum_{\beta\in\mathcal{A}}L(\alpha,\beta)u(t,\beta),\qquad t\in[0,1],\quad\alpha\in\mathcal{A}
\end{equation}
where $\mathcal{A}$ and $L$ are as in section 2.2. Let $\nu_0\in\mathcal{C}$, if for some $\epsilon>0$,
\begin{equation}\label{udecay}
|u(t,\nu_0(k))|\leq C\Big(\frac{\E}{(2+\epsilon)k}\Big)^k,\qquad k>0,\quad t\in\{0,1\}
\end{equation}
then $u(t,\nu_0(k))=0$ for all $t\in[0,1]$ and $k\geq 0$.
\end{Th}
\begin{proof}To prove this result we will follow a similar strategy as in \cite{JLMP} in theorem 2.3.\\ \\
Let $L$ be the operator of \eqref{evouproblem}, $L:\l^2(\mathcal{A})\to\l^2(\mathcal{A})$, then the solution $u(t,\alpha)$ is defined by
\begin{equation*}
u(\cdot,t)=\E^{\I Lt}u(\cdot,0)
\end{equation*}
and hence $(u(t,\alpha))_{\alpha\in\mathcal{A}}$ is in $\l^2(\mathcal{A})$ for all $t\in[0,1]$.\\ \\
Consider the auxiliar function $\Phi$ 
\begin{equation}\label{Phi1}
\Phi(t,\theta)=\sum_{\alpha\in\mathcal{A}} u(t,\alpha)\psi_{\nu_0}(\alpha,\theta)
\end{equation}
and suppose that $\Phi\neq 0$ to get a contradiction, where $\psi_{\nu_0}(\alpha,\theta)$ is defined as in \eqref{varphialpha} with $\textbf{n}=S^{-1}(\theta)(\delta_{\nu_0}(\sigma))_{\sigma\in\mathcal{C}}$, here $\delta_{\nu_0}(\cdot)$ is the Kronecker delta. Thus
\begin{equation*}
\begin{split}
-\I\partial_t\Phi(t,\theta)&=-\I\partial_t\sum_{\alpha\in\mathcal{A}}u(t,\alpha)\psi_{\nu_0}(\alpha,\theta)=-\I\langle\partial_t(u(t),\psi_{\nu_0}(\theta)\rangle\\
&=\langle u(t),\lambda(\theta)\psi_{\nu_0}(\theta)\rangle=\lambda(\theta)\Phi(t,\theta)
\end{split}
\end{equation*}
where $\lambda(\theta)$ is defined in \eqref{lambdapara}. In particular we have obtained
\begin{equation}\label{Phi2}
\Phi(t,\theta)=\Phi(0,\theta)\E^{\I\lambda(\theta)t}
\end{equation}
On the other hand, using the definition of $\psi_{\nu_0}(\alpha,\theta)$,  \eqref{Phi1} can be rewritten as 
\begin{equation*}
\begin{split}
\Phi(t,\theta)&=\sum_{\alpha\in\mathcal{A}}u(t,\alpha)\psi_{\nu_0}(\alpha,\theta)\\
&=\sum_{\alpha\in\mathcal{A}_1}u(t,\alpha)\psi_{\nu_0}(\alpha,\theta)+\sum_{\alpha_\in\mathcal{A}_0}u(t,\alpha)\psi_{\nu_0}(\alpha,\theta)\\
&=\Big(\sum_{\alpha\in\mathcal{A}_1}u(t,\alpha)\psi_{\nu_0}(\alpha,\theta)+\sum_{\nu\in\mathcal{C}}\sum_{k\geq 1}^{K_0}u(t,\nu(k))\psi_{\nu_0}(\nu(k),\theta)+\\
&\sum_{\nu\in\mathcal{C}}\sum_{k>K_0}u(t,\nu(k))s^{-1}(\nu,\nu_0;\theta)\theta^{-k}\Big)+\Big(\sum_{k>K_0}u(t,\nu_0(k))\theta^k\Big)\\
&=A(t,\theta)+B(t,\theta)
\end{split}
\end{equation*}
here $s^{-1}(\nu,\sigma;\theta)$ denotes the entries of the matrix $S^{-1}(\theta)$. The functions $\Phi(t,\theta)$ are in $L^2(\mathbb{T})$ for all $t\in[0,1]$. Moreover, by lemma \ref{Lemmarational} we have that $A(t,\theta)$ converges for $|\theta|\geq1$ and $B(t,\theta)$ for $|\theta|\leq1$. For $t=0$ and $t=1$, $B(t,\theta)$ also converges for $|\theta|>1$, due to \eqref{udecay}. Thus $\Phi(0,\theta)$ and $\Phi(1,\theta)$ are analytic functions in $\mathbb{C}\setminus\overline{\mathbb{D}}$, except maybe at
$\theta\in\mathscr{O}$ (see \eqref{Odef}). Thus, using Corollary 3.2 in \cite{JLMP}, one can extend this convergence property to $\Phi(t,\theta)$ for all $t\in[0,1]$.\\
Now, using \eqref{sumasigma}
\begin{equation*}
\limsup_{|\theta|\to\infty}\frac{\log|A(t,\theta)|}{|\theta|}\leq\max\{\sigma_f,\sigma_g\}
\end{equation*}
where
\begin{equation*}
\begin{split}
&f(\theta)=\sum_{\alpha\in\mathcal{A}_1}u(t,\alpha)\psi_{\nu_0}(\alpha,\theta)+\sum_{\nu\in\mathcal{C}}\sum_{k\geq 1}^{K_0}u(t,\nu(k))\psi_{\nu_0}(\nu(k),\theta)\\
&g(\theta)=\sum_{\nu\in\mathcal{C}}\sum_{k>K_0}u(t,\nu(k))s^{-1}(\nu,\nu_0;\theta)\theta^{-k}=\sum_{\nu\in\mathcal{C}}o(\theta^{-K_0})s^{-1}(\nu,\nu_0;\theta)
\end{split}
\end{equation*}
Notice that by lemma \ref{Lemmarational} $f$ and $s^{-1}(\nu,\nu_0;\theta)$ are rational functions with respect to $\theta$, whence $\sigma_f,\sigma_g\leq 0$ and
\begin{equation}\label{sigmaA}
\limsup_{|\theta|\to\infty}\frac{\log|A(t,\theta)|}{|\theta|}\leq 0
\end{equation}
It follows from \eqref{udecay} and theorem \ref{Thtypeformula} that $B(t,\theta)=\sum_{k>K_0}u(t,\nu_0(k))\theta^k$ are entire functions of exponential type at most $(2+\epsilon)^{-1}$ for $t\in\{0,1\}$. In particular,
\begin{equation*}
\limsup_{r\to\infty}\frac{\log|\Phi(t,r\E^{\I\varphi})|}{r}\leq\frac{\log|B(t,r\E^{\I\varphi})|}{r}\leq\frac{1}{2+\epsilon},\qquad t\in\{0,1\},\quad\varphi\in[0,2\pi]
\end{equation*}
whence, on one hand by \eqref{hineq} we have
\begin{equation*}
\begin{split}
0&\leq \limsup_{r\to\infty}\frac{\log|\Phi(t,r\E^{\I\pi/2})|}{r}+\limsup_{r\to\infty}\frac{\log|\Phi(t,r\E^{-\I\pi/2})|}{r}\\
&\leq\limsup_{r\to\infty}\frac{\log|\Phi(t,r\E^{\pm \I\pi/2})|}{r}+\frac{1}{2+\epsilon},\qquad t\in\{0,1\}
\end{split}
\end{equation*}
And on the other hand, using \eqref{Phi2}
\begin{equation*}
\begin{split}
\limsup_{y\to\infty}\frac{\log|\Phi(1,\I y)|}{y}&=1+\limsup_{y\to\infty}\frac{\log|\Phi(0,\I y)|}{y}\\
&\geq 1-\frac{1}{2+\epsilon}>\frac{1}{2+\epsilon}
\end{split}
\end{equation*}
Thus, we have a contradiction unless $\Phi\equiv 0$.\\ \\
\emph{We claim that $\Phi\equiv 0$ implies $B(t,\theta)=0$ for all $t\in[0,1]$}\\ \\
Suppose that $B(t,\theta)\neq 0$ to get a contradicttion. Since $\Phi\equiv0$, we have that $A(t,\theta)=-B(t,\theta)$, and by lemma \ref{lemmaUpoles} there exists a polinomial $P(z)=\sum_{j=0}^Np_jz^j$, with all its roots $a_j$ simple and $|a_j|<1$ such that $P(\theta)A(t,\theta)=\sum_{k=-\infty}^{N_0}d_k\theta^k$ and  $P(\theta)B(t,\theta)=\sum_{k> K_0}c_k\theta^k$. This implies that for $k\gg K_0$, $c_k=0$. That is, let $\mathscr{P}$ be the matrix 
\begin{equation*}
\mathscr{P}=\begin{pmatrix}0&-p_{N-1}&-p_{N-2}&\dots&-p_1&-p_0&0&0&\dots\\
0&0&-p_{N-1}&-p_{N_2}&\dots&-p_1&-p_0&0&\dots\\
\vdots&\vdots&\ddots&\ddots&\ddots&\ddots&\ddots&\ddots&\ddots\end{pmatrix}
\end{equation*}
Then the condition $c_k=0$, $k\gg K_0$, is equivalent to $\mathscr{P}(u(t,\nu_0(k)))_{k\gg K_0}=(u(t,\nu_0(k)))_{k\gg K_0}$. Using \eqref{sigmaA} and \eqref{hineq} we get $\sigma_{B(t,\theta)}=0$, which is a contradiction with the fact that $c_k=0$, unless $u(t,\nu_0(k))=0$ for $k\gg K_0$ and hence for all $k\geq 0$.
\end{proof}


\section{Uniqueness of the solution of the Schr\"odinger equation on finite graphs and applications}\label{section4}

In general it is not true that a solution $u(t,\alpha)$ of \eqref{evouproblem} is trivial if it is zero on one channel. Thus, the aim of this section is to study what happens on $\mathcal{A}_1$ in such cases. \\ \\
If we consider $\mathcal{F}$ as  a finite set of particles and we study small oscillations around its equilibrium position, see \cite{LM2,MS}, then the problem is reduced to  the spectral problem
\begin{equation*}
\lambda x(\alpha)=\sum_{\eta\in\mathcal{F}}L(\alpha,\eta)x(\eta),\quad \alpha\in\mathcal{F}
\end{equation*}
Here $L(\cdot,\cdot)$ is a symmetric real valued  and positive matrix. There is a graph $\mathcal{G}=(\mathcal{F},\Pi)$ associated to this problem. Here the elements of $\mathcal{F}$ denote the vertices of the graph and $\Pi$ is the set of the edges. It is given by the matrix $L(\cdot,\cdot)$, that is, there is an edge $(\alpha,\beta)\in\Pi$ if $\alpha\neq\beta$ and $L(\alpha,\beta)\neq 0$.\\ \\
In the sequel we will use $\mathcal{F}$ to denote a finite set of particles and $\mathcal{A}$ to denote sets as in section \ref{section22}.\\
Now if we look at the dynamics of this problem, that is
\begin{equation}\label{eqgraph}
\partial_t u(t,\alpha)=\I\sum_{\eta\in\mathcal{F}}L(\alpha,\eta)u(t,\alpha),\quad \alpha\in\mathcal{F}, \quad t\in[0,1]
\end{equation}
Then the solution $u(t,\alpha)$ of \eqref{eqgraph} is given by
\begin{equation}
u(t)=\E^{\I tL}u(0)
\end{equation}
here $u(0)=(u(0,\alpha))_{\alpha\in\mathcal{F}}$ and the dimension of the vector space $V$ formed by all solutions of \eqref{eqgraph} is dim($V)=\sharp\mathcal{F}$. A natural question is:  what happens if for some $\alpha\in\mathcal{F}$, $u(t,\alpha)=0$, $t\in[0,1]$? Is it true that $u(t,\beta)=0$ for all $\beta\in\mathcal{F}$ and $t\in[0,1]$? And if this is not the case, then  how big is dim($V$) with these extra conditions?\\ \\
Consider the following example to illustrate this problem
\begin{figure}[H]
\begin{displaymath}
\xymatrix{\beta_1\ar@{-}[r]\ar@{-}[d]&\alpha_1\ar@{-}[d]\\ \alpha_2\ar@{-}[r]&\beta_2}
\end{displaymath}\caption{}\label{fig5}
\end{figure} 
Here $u(t,\beta_j)=0$, $j=1,2$, $t\in[0,1]$. Thus we have
\begin{equation}\label{eqgraphboundary}
\partial_t u(t,\alpha_j)=\I L(\alpha_j,\alpha_j)u(t,\alpha_j),\quad\text{whence }\quad u(t,\alpha_j)=\E^{\I tL(\alpha_j,\alpha_j)}u(0,\alpha_j)
\end{equation}
here $j=1,2$, then by the boundary condition  $u(t,\beta_j)=0$, and the equation \eqref{eqgraph} at $\beta_j$, $j=1,2$ 
\begin{equation}\label{eqgraph2}
0=L(\beta_j,\alpha_1)u(t,\alpha_1)+L(\beta_j,\alpha_2)u(t,\alpha_2),\quad j=1,2
\end{equation}
Now using the expresion \eqref{eqgraphboundary} into \eqref{eqgraph2}
\begin{equation*}
0=\E^{\I tL(\alpha_1,\alpha_1)}L(\beta_j,\alpha_1)u(0,\alpha_1)+\E^{\I tL(\alpha_2,\alpha_2)}L(\beta_j,\alpha_2)u(0,\alpha_2)
\end{equation*}
And there are two options either $L(\alpha_1,\alpha_1)\neq L(\alpha_2,\alpha_2)$ which implies
\begin{equation*}
L(\beta_j,\alpha_k)u(0,\alpha_k)=0,\quad j,k=1,2
\end{equation*}
But there is an edge from $\alpha_k$ to $\beta_j$, thus $L(\beta_j,\alpha_k)\neq 0$ and $u(0,\alpha_k)=0$, in particular $u(t)= 0$, $t\in[0,1]$ and dim($V$)=0.\\ \\
Either $L(\alpha_1,\alpha_2)=L(\alpha_2,\alpha_2)$ then
\begin{equation*}
\begin{cases}0=&L(\beta_1,\alpha_1)u(0,\alpha_1)+L(\beta_1,\alpha_2)u(0,\alpha_2)\\
0=&L(\beta_2,\alpha_1)u(0,\alpha_1)+L(\beta_2,\alpha_2)u(0,\alpha_2)\end{cases}
\end{equation*}
And dim$(V)=0,1$ depending on the rank of the matrix $(L(\beta_j,\alpha_k))_{j,k=1,2}$ . \\ \\
The generalization of this result is shown in  theorem \ref{Thuniqgraph}. Before we formulate it, we need some definitions and
%
%
%
%
the concept of extension of a subgraph, given in \cite{LM2,MS}. We will use their notation as well. In order that our article is self-contained, we repeat it here.\\ \\
Let $\mathcal{F}$ be a finite set of points and let $\mathcal{G}=(\mathcal{F},\Pi)$ be a connected graph formed by the set $\mathcal{F}$ and the edges $(\alpha,\beta)\in\Pi$, where $\alpha,\beta\in\mathcal{F}$. Given a set $\mathcal{B}\subset\mathcal{F}$, we want to extend this set to a bigger one as follows:\\
Let $\beta\in\mathcal{B}$ be such that there exists a unique  $\alpha\in\mathcal{F}\setminus\mathcal{B}$ with $(\alpha,\beta)\in\Pi$
\begin{figure}[H]
\begin{displaymath}
\xymatrix{\bullet_\beta\ar@{-}[r] &\circ_\alpha\ar@{-}[r]&\circ\ar@{-}[r]&\bullet}
\end{displaymath}\caption{}
\end{figure}
Here $(\bullet)$ denotes the elements of $\mathcal{B}$. Then we say that $\mathcal{B}^{(1)}=\mathcal{B}\cup\{\alpha\}$ is an \emph{extension} of $\mathcal{B}$. We can iterate this process: $\mathcal({B}^{(k)})^{(1)}=\mathcal{B}^{(k+1)}$ and we will obtain a chain of prolongations
\begin{equation*}
\mathcal{B}\subset\mathcal{B}^{(1)}\subset\dots\subset\mathcal{B}^{(p)}
\end{equation*}
If $\mathcal{B}^{(p)}$ does not have an extension, we say that it is \emph{maximal} and we denote it by $[\mathcal{B}]$, this set depends only on $\mathcal{B}$ as it is shown in the following lemma, see \cite{LM2} and \cite{MS}:
\begin{Lemma}\label{Luniqueness}
Given a subset $\mathcal{B}\subset\mathcal{F}$, all maximal chains that begins at $\mathcal{B}$ end with the same set $[B]\subset\mathcal{F}$.
\end{Lemma}
In addition of the previous concepts from \cite{LM2,MS}, we  give some new ones that we need afterwards to set our main result of this section, theorem \ref{Thuniqgraph}.\\ \\
Remember that our graph was connected, thus we can consider the connected components of the graph $\mathcal{G}^\prime$ which results from $\mathcal{F}\setminus[\mathcal{B}]$. We will call each connected component of $\mathcal{G}^\prime$ a \emph{branch} and for each $\alpha\in\mathcal{F}\setminus[\mathcal{B}]$ we will denote them by $\gamma_\alpha$, where $\alpha\in\gamma_\alpha$. Notice that  $\gamma_\alpha=\gamma_{\alpha^\prime}$ if and only if $\alpha\in\gamma_{\alpha^\prime}$ and $\alpha^\prime\in\gamma_\alpha$
\begin{figure}[H]
\begin{displaymath}
\xymatrix{\bullet\ar@{-}[r]&\circ\ar@{-}[r]&\circ} \qquad \qquad \xymatrix{\ast\ar@{-}[r]&\ast\ar@{-}[r]&\ast}
\end{displaymath}\caption{}\label{graph1}
\end{figure}
Here $(\ast)$ denotes the elements of $[\mathcal{B}]$ and it can be observed that  $[\mathcal{B}]=\mathcal{F}$, which in particular implies that there  are no branches.\\ \\
Given a set $\mathcal{B}\subset\mathcal{F}$ we have defined the concept of branch, which depends on $[\mathcal{B}]$, so the natural question is if there is some notion which allow us to gather the branches. Thus we define the \emph{cluster}. Let $\beta\in[\mathcal{B}]$, we call $(\beta)$ a cluster and
\begin{equation*}
(\beta)=\cup_{\alpha\in J}\cup_{\xi\in\gamma_\alpha}\{\xi\},\quad J=\{\alpha\in\mathcal{F}\setminus[\mathcal{B}]:\text{ } (\alpha,\beta)\in\Pi\}
\end{equation*}
In other words, the cluster $(\beta)$ is the set of the particles which form the branches that are attached to $\beta$.\\ \\ 
It happens that a branch is attached to $n$ different clusters, $n\geq 1$, then we say that the branch is of \emph{order }$n-1$, i.e., ord$(\gamma_\alpha)=n-1$, for $\alpha\in\mathcal{F}\setminus[\mathcal{B}]$.
\begin{figure}[H]
\begin{displaymath}
\xymatrix{\circ \ar@{-}[dr]& &\circ\ar@{-}[dl]& &\circ\ar@{-}[dl]\\
&\bullet\ar@{-}[dl]\ar@{-}[dr]\ar@{-}[r]&\circ_\alpha\ar@{-}[r]&\bullet\ar@{-}[dr]& \\
\circ& &\circ& &\circ }\qquad\qquad \xymatrix{\circ \ar@{-}[dr]& &\circ\ar@{-}[dl]& &\circ\ar@{-}[dl]\\
&\ast\ar@{-}[dl]\ar@{-}[dr]\ar@{-}[r]&\circ_\alpha\ar@{-}[r]&\ast\ar@{-}[dr]& \\
\circ& &\circ& &\circ }
\end{displaymath}\caption{}\label{graph2}
\end{figure}
Here $\mathcal{B}=[\mathcal{B}]$, there are two different clusters  and each point of $\mathcal{F}\setminus[\mathcal{B}]$ forms a branch. Each branch is of order zero, except the one formed by the point $\alpha$ which is of order one, since it belongs to two different cluster as we can see in the following figure
\begin{figure}[H]
\begin{displaymath}
\xymatrix{\circ\ar@{-}[dr]& &\circ\ar@{-}[dl]\\
 &\ar@{-}[r]&\circ_\alpha\\
\circ\ar@{-}[ur]& &\circ\ar@{-}[ul]}_\text{cluster 1}\qquad\qquad\xymatrix{&&\circ\ar@{-}[dl]\\
\circ_\alpha\ar@{-}[r]& \\
&& \circ\ar@{-}[ul]}_\text{cluster 2}
\end{displaymath}\caption{}
\end{figure}
%
%
%
%

\begin{Th}\label{Thuniqgraph}
The dimension of the vector space $V_{\mathcal{B}}$ formed by  the solutions $u(t,\alpha)$ of \eqref{eqgraph} such that $u(t,\beta)=0$, $\beta\in\mathcal{B}$, $t\in[0,1]$ is bounded as follows
\begin{equation}\label{thineq}
\sharp(\mathcal{F}\setminus[\mathcal{B}])-\sum_{i=1}^M\mathfrak{K}_i\leq \text{dim}(V_{\mathcal{B}})\leq\sharp(\mathcal{F}\setminus[\mathcal{B}])+\sum_{i=1}^N\text{ord}(\gamma_i)\mathfrak{N}_i-\sum_{i=1}^M\mathfrak{K}_i
\end{equation}
where M denotes the number of different clusters with respect to $[\mathcal{B}]$, N the number of different branches $\gamma_i$ with respect to $[\mathcal{B}]$, $\mathfrak{K}_i$ is the number of different eigenvalues that comes from the restriction of the matrix $L(\cdot,\cdot)$ in \eqref{eqgraph} to the cluster $(\beta_i)$ and $\mathfrak{N}_i$ is the number of different eigenvalues that comes from the restriction of the matrix $L(\cdot,\cdot)$ in  \eqref{eqgraph} to the branch  $\gamma_{i}$.
\end{Th}
If we look at the example in figure \ref{graph2}, the theorem tells us that $-1\leq\text{dim}(V_{\mathcal{B}})\leq 0$ if all the eigenvalues are different, thus dim($V_\mathcal{B})=0$, i.e, $u\equiv 0$.\\
For the example in figure \ref{fig5}  we have that if the eigenvalues are different then $-2\leq\text{dim}(V_{\mathcal{B}})\leq 0$, which means that $\text{dim}(V_{\mathcal{B}})=0$ and if there is only one eigenvalue then $0\leq\text{dim}(V_{\mathcal{B}})\leq 2$.\\ \\
Before proving the theorem we need a technical lemma
\begin{Lemma}\label{lemmaaux}
Let $u(t,\alpha)$ be a solution of \eqref{eqgraph} and let $\mathcal{B}\subset\mathcal{F}$  be such that
\begin{equation*}
u(t,\beta)=0,\quad \beta\in\mathcal{B},\quad t\in[0,1]
\end{equation*}
Then $u(t,\beta)=0$ for all $\beta\in[\mathcal{B}]$, $t\in[0,1]$.
\end{Lemma}
\begin{proof}
Consider a chain of prolongations $\mathcal{B}^{(j)}$ of $\mathcal{B}$. Let $\{\alpha_1\}=\mathcal{B}^{(1)}\setminus\mathcal{B}$ and  $\beta\in\mathcal{B}$ be such that $L(\alpha_1,\beta)\neq 0$, then by \eqref{eqgraph}
\begin{equation*}
\begin{split}
0=&\partial_tu(t,\beta)=\I\sum_{\alpha\in\mathcal{F}}L(\beta,\alpha)u(t,\alpha)=\\
&\I(\sum_{\alpha\in\mathcal{B}}L(\beta,\alpha)u(t,\alpha)+\sum_{\alpha\in\mathcal{F}\setminus\mathcal{B}}L(\beta,\alpha)u(t,\alpha))=\\
&\I(0+L(\alpha_1,\beta)u(t,\alpha_1)),\quad t\in[0,1]
\end{split}
\end{equation*}
Thus $u(t,\alpha_1)=0$ and applying an inductively argument, it follows that for all $j\geq 1$, $u(t,\beta)=0$, $\beta\in\mathcal{B}^{(j)}$, $t\in[0,1]$. \\
Thus, by lemma \ref{Luniqueness}, after some $j_0$, $\mathcal{B}^{(j_0)}=[\mathcal{B}]$ and we obtain $u(t,\beta)=0$, $\beta\in[\mathcal{B}]$.
\end{proof}
\begin{proof}[Proof of Theorem~\ref{Thuniqgraph}.]
First of all notice that due to the previous lemma  dim$(V_{\mathcal{B}})=\text{dim}(V_{[\mathcal{B}]})$. \\
The problem \eqref{eqgraph} can be splitted into different independent pieces, that is the study of \eqref{eqgraph} restricted to each of its cluster. Thus, let $\beta\in[\mathcal{B}]$ and consider the restriction of \eqref{eqgraph} to the cluster $(\beta)\neq\emptyset$, that is:
\begin{equation}\label{lbeta}
\partial_t u(t,\alpha)=\I\sum_{\xi\in(\beta)}L(\alpha,\xi)u(t,\xi),\quad t\in[0,1],\quad \alpha\in(\beta)
\end{equation}
In what follows, to simplify notation, we denote by $L_{(\beta)}$ the matrix $L(\cdot,\cdot)$ of \eqref{eqgraph} restricted to the cluster $(\beta)$, in other words, $L_{(\beta)}$ is the matrix of \eqref{lbeta}.\\
For each $\alpha\in(\beta)$ we associate a number $j(\alpha):=j$, $1\leq j\leq n$, here $n=\sharp(\beta)$ and we will write $j$ instead of $\alpha$.\\
General solution of \eqref{lbeta} can be written in a matrix form as
\begin{equation}\label{u1}
u(t)=e^{itL_{(\beta)}}u(0)
\end{equation}
If we denote $P$ the matrix of the eigenvectors of $L_{(\beta)}$, then $P^{-1}L_{(\beta)}P=\textbf{diag}(\lambda_j)_{j=1}^n$, where $\lambda_1\leq\lambda_2\leq\dots\leq\lambda_n$ are the real  eigenvalues of $L_{(\beta)}$. This happens because $L(\cdot,\cdot)$ is a symmetric real matrix whence $L_{(\beta)}$ is real and symmetric as well. If $y_0=P^{-1}u(0)$, then \eqref{u1} turns into
\begin{equation*}
u(t)=\sum_{j=1}^n \E^{\I t\lambda_j}y_{0,j}p_j
\end{equation*}
where $y_{0,j}$ denotes de $j-$entry of the vector $y_0$ and the  $p_j$ are the columns of the matrix $P$, that is, the eigenvectors of $L_{(\beta)}$.\\ \\
Without lost of generality, we assume that $\{p_j\}_{j=1}^n$ forms an orthogonal system.\\ \\
Using now the boundary condition of the cluster $(\beta)$, that is $u(t,\beta)=0$, $t\in[0,1]$, we have
\begin{equation}\label{u3}
0=\sum_{j=1}^nL(\beta,j)u(t,j)
\end{equation}
Let $\{\mu_j\}_{j=1}^{\mathfrak{K}}$ be the set of eigenvalues of the matrix $L_{(\beta)}$ without any repetitions, that is $\mu_1<\mu_2<\dots<\mu_{\mathfrak{K}}$, then equation \eqref{u3} turns into
\begin{equation*}
\begin{split}
0=&\sum_{j=1}^nL(\beta,j)u(t,j)=\sum_{j=1}^n L(\beta,j)\sum_{m=1}^n\E^{\I t\lambda_m}y_{0,m}p_m(j)=\\
&\sum_{m=1}^n\E^{\I t\lambda_m}y_{0,m}\sum_{j=1}^nL(\beta,j)p_m(j)=\sum_{m=1}^\mathfrak{K}\E^{\I t\mu_m}\sum_{s\in \tilde{m}}y_{0,s}\sum_{j=1}^nL(\beta,j)p_s(j)
\end{split}
\end{equation*}
where $\tilde{m}:=\{j:\text{ }\lambda_j=\nu_m\}$ and $p_s(j)$ denotes the entry of the matrix $P$ which corresponds to the j-row and s-column.\\
In particular we have obtained 
\begin{equation}\label{A1}
\begin{split}
0=&\sum_{s\in\tilde{m}}y_{0,s}\sum_{j=1}^nL(\beta,j)p_s(j)=\sum_{s\in\tilde{m}}\big(\sum_{l=1}^n\{p_l^{-1}(s)\sum_{j=1}^nL(\beta,j)p_s(j)\}u(0,l)\big)\\
&=\sum_{l=1}^n\{\sum_{s\in\tilde{m}}p_l^{-1}(s)\sum_{j=1}^nL(\beta,j)p_s(j)\}u(0,l)\qquad 1\leq m\leq\mathfrak{K}
\end{split}
\end{equation}
Here $\{p^{-1}_l(s)\}_{l,s=1}^n$ denotes the entries of the matrix $P^{-1}$.\\
Thus equation \eqref{A1} can be written in matrix form as $A_\beta u(0)=0$, where $A_\beta$ is a $\mathfrak{K}\times n$ matrix. The rows $a(m)$ of matrix $A_\beta$ fullfill
\begin{equation}\label{defAbeta}
a(m)=\sum_{s\in\tilde{m}}p^{-1}(m)\sum_{j=1}^nL(\beta,j)p_s(j),\quad 1\leq m\leq \mathfrak{K}
\end{equation}
\emph{We claim that $\text{rank}(A_\beta)=\mathfrak{K}$.}\\ \\
Let $m$, $1\leq m\leq\mathfrak{K}$, and $s\in\tilde{m}$ be such that 
\begin{equation}\label{lcero}
\sum_{j=1}^nL(\beta,j)p_s(j)\neq 0
\end{equation}
otherwise, $L(\beta,j)=0$ for $1\leq j\leq n$, since the vectors $p_j$, $1\leq j\leq n$ form an orthogonal base for $\mathbb{R}^n$. Thus, if $L(\beta,j)=0$ for $1\leq j\leq n$ then $(\beta)=\emptyset$ which is a contradiction.\\
Assume that $\text{rank}(A_\beta)<\mathfrak{K}$ and pick $m,s$ such that the equation \eqref{lcero} is fullfilled. There are constants $\eta_j$, not all of them zero, such that 
\begin{equation*}
a(m)=\sum_{j=1,j\neq m}^\mathfrak{K}\eta_ja(j)
\end{equation*}
whence
\begin{equation*}
0=\sum_{j=1}^n\eta_j^\prime p^{-1}(j)
\end{equation*}
where
\begin{equation*}
\eta_j^\prime=\begin{cases}\eta_k\sum_{l=1}^nL(\beta,l)p_j(l)\quad&\text{if }j\in\tilde{k},\text{ }k\neq m\\
-\sum_{l=1}^nL(\beta,l)p_j(l)\quad &\text{if }j\in\tilde{m}\end{cases}
\end{equation*}
In particular, by the election of $m,s$, not all $\eta_j^{\prime}$ are zero, which is a contradiction with the fact that rank$(P^{-1})=n$.\\ \\
Define the matrix $T_{\sum_{i=1}^M\mathfrak{K}_i\times\sharp(\mathcal{F}\setminus[\mathcal{B}])}$ as
\begin{equation}\label{mT}
T_\alpha(j)=\begin{cases}a_\alpha^{(\beta_j)}(M_j-j),\quad &\text{if }\alpha\in(\beta_j)\text{ and }1\leq M_j-j\leq\mathfrak{K}_{\beta_j}\\
0\quad&\text{otherwise}\end{cases}
\end{equation}
Here $j$ is bounded by $\sum_{i=1}^n\mathfrak{K}_i\leq j<\sum_{i=1}^{n+1}\mathfrak{K}_i$, thus  $M_j=\sum_{i=1}^{n+1}\mathfrak{K}_i$, $(\beta_j)=(\beta_n)$ and $\mathfrak{K}_{\beta_j}=\mathfrak{K}_n$. Finally, $a_\alpha^{(\beta_j)}(\cdot)$ are the entries of the matrices $A_{\beta_j}$ defined in \eqref{defAbeta}.\\ \\
Now dim$(V_\mathcal{B})=\text{dim}(Ker(T))$ and $\text{rank}(T)\leq\sum_{i=1}^M\mathfrak{K}_i$, using $\sharp(\mathcal{F}\setminus[\mathcal{B}])=rank(T)+dim(Ker(T))$, the left inequality of \eqref{thineq} follows.\\ \\
To show the right inequality of \eqref{thineq},  first consider the matrices $A_{\beta_j}^\prime$, $1\leq j\leq M$.  Define $A^\prime_{\beta_1}=A_{\beta_1}$. For $j>1$ choose the branches $\gamma_i\subset(\beta_j)$, such that $ord(\gamma_i)\geq 1$ and $\gamma_i\subset(\beta_n)$, for some $n<j$ .\\
Fix the set $\Theta=\{\mu_m\}$, where $\mu_m$ are the eigenvalues which comes from those branches $\gamma_i$, i.e., the eigenvalues of $L_{\gamma_i}$.\\
Thus the rows $a^{\prime,(\beta_j)}(s)$ of the matrix $A_{(\beta_j)}^\prime$ are
\begin{equation*}
a^{\prime,(\beta_j)}\begin{cases}a^{(\beta_j)}(s),\quad&\text{if }\mu_s\notin\Theta\\
0\quad&\text{otherwise}\end{cases}
\end{equation*}
Define the matrix $T^{\prime}$ as we did for the matrix $T$ in \eqref{mT} but with the matrices $A^\prime_{(\beta_j)}$. We have obtained that if $u(t)\in V_{\mathcal{B}}$, then $u(t)\in Ker(T^\prime)$ by construction, whence
\begin{equation}\label{Tprime1}
dim(V_{\mathcal{B}})\leq dim(Ker(T^\prime))=\sharp(\mathcal{F}\setminus[\mathcal{B}])-rank(T^\prime)
\end{equation}
By construction, the matrix $T^\prime$ is diagonal block matrix where each block which is nonzero is equal to $A^\prime_{(\beta_j)}$, $1\leq j\leq M$. Thus 
\begin{equation}\label{Tprime2}
rank(T^\prime)=\sum_{j=1}^Mrank(A^\prime_{(\beta_j)})=\sum_{j=1}^M\mathfrak{K}_j-\sum_{j=1}^Nord(\gamma_j)\mathfrak{N}_j
\end{equation}
Where $N$ denotes the total number of different branches of $\mathcal{F}\setminus[\mathcal{B}]$, $\gamma_j$ are the branches of the system and $\mathfrak{N}_j$ the corresponding different number of eigenvalues of $L_{\gamma_j}$.\\
Using equation \eqref{Tprime2} in \eqref{Tprime1} the theorem is proved
\end{proof}
\begin{Cor}
In the hypothesis of theorem \ref{Thuniqgraph},
\begin{equation}
dim(V_\mathcal{B})=dim(\cap_{j=1}^MKer(A_{\beta_j}))
\end{equation}
where $M$ is the number of the clusters of the system with respect to $[\mathcal{B}]$ and the matrices $A_{\beta_j}$ are defined in \eqref{defAbeta}.
\end{Cor}
We give now  some results on the uniqueness of the solution $u(t,\alpha)$, $\alpha\in\mathcal{A}$ of the problem \eqref{evouproblem} depending on the zeros of $u(t,\alpha)$. We use the idea given in \cite{MS} in chapter 17, where they extend $\mathcal{A}_0$ to the whole system, that is $[\mathcal{A}_0]=\mathcal{A}$.\\ \\
If the zeros are in $\mathcal{A}_1$, then corollary  \ref{cor42} extends the solution on the channels uniquely. If the solution is zero on the channels, then we use corollary \ref{cor43}
\begin{Cor}\label{cor42}
Let $u(t,\alpha)$ be a solution of \eqref{evouproblem}. Set $\mathcal{F}=\mathcal{A}_1\cup\{\nu(1)\}_{\nu\in\mathcal{C}}$. If there exists some subset $\mathcal{B}\subset\mathcal{F}$, such that $[\mathcal{B}^\prime]=\mathcal{F}$ and $u(t,\beta)=0$ for all $\beta\in\mathcal{B}^\prime$, $t\in[0,1]$, where $\mathcal{B}^\prime=\mathcal{B}\setminus\{\nu(1):\text{ }\nu(1)\in\mathcal{B}\text{ and }\nu(0)\notin\mathcal{B},\text{ }\nu\in\mathcal{C}\}$.\\
Then $u(t,\alpha)=0$ for all $t\in[0,1]$ and $\alpha\in\mathcal{A}$.
\end{Cor}
\begin{proof}
A simple application of lemma \ref{lemmaaux} shows us that $u(t,\beta)=0$ for all $t\in[0,1]$ and $\beta\in\mathcal{F}$. In particular for any $\nu\in\mathcal{C}$ we have obtained that $u(t,\nu(1))=u(t,\nu(0))=0$, $t\in[0,1]$, and hence $u(t,\nu(k))=0$, for all $k\geq 0$ and $t\in[0,1]$.
\end{proof}
\begin{Cor}\label{cor43}
Let $u(t,\alpha)$ be a solutoin of \eqref{evouproblem}, such that for all $\nu\in\mathcal{C}\setminus\{\nu_0\}$ and some $\epsilon>0$ 
\begin{equation*}
|u(t,\nu(k))|\leq C\Big(\frac{\E}{(2+\epsilon)k}\Big)^k,\qquad k>0,\quad t\in\{0,1\}
\end{equation*}
If in addition $[\mathcal{B}]=\mathcal{F}$, where $\mathcal{B}=\cup_{j=0,1}\cup_{\nu\in\mathcal{C}\setminus\{\nu_0\}}\{\nu(j)\}$ and $\mathcal{F}=\mathcal{A}_1\cup\{\nu(1)\}_{\nu\in\mathcal{C}}$, then $u(t,\alpha)=0$ for all $\alpha\in\mathcal{A}$ and $t\in[0,1]$.
\end{Cor}
\begin{proof}
By theorem \ref{thdecayuniq}, $u(t,\nu(k))=0$ for all $\nu\neq\nu_0$, $\nu\in\mathcal{C}$ and $k\geq 0$, $t\in[0,1]$. Thus $u(t,\beta)=0$ for all $\beta\in\mathcal{B}$, $t\in[0,1]$ and using corollary \ref{cor42} the result follows
\end{proof}
\begin{Rem}
\begin{enumerate}
\item
Notice that  in the proof of the theorem \ref{Thuniqgraph} it is not necessary that the matrix $L(\cdot,\cdot)$ has  to be positive, only symmetric and real valued.
\item
In corollary \ref{cor43}, if the solution $u(t,\alpha)$ is zero along all the channels $\nu\in\mathcal{C}$, then $u(t,\alpha)$ will be trivial if $[\cup_{\nu\in\mathcal{C}}\{\nu(0)\}]=\mathcal{F}$, where $\mathcal{F}=\mathcal{A}_1$.
\end{enumerate}
\end{Rem}

\medskip

\noindent 
{\bf Acknowledgments.}
I am indebted to  Yura Lyubarskii  to set the problem and  for productive discussions on this topic.


\begin{thebibliography}{9}

\bibitem{ART}
I. Alvarez-Romero, G. Teschl, \emph{ A Dynamic Uncertainty Principle for Jacobi Operators}, 	\arxiv{1608.04344}

\bibitem{BK}
G. Berkolaiko, P. Kuchment, \emph{ Introduction to quantum graphs}, Mathematical Surveys and Monographs, vol 186. Amer. Math. Soc., Providence, RI, 2013, pp. xiv+270

\bibitem{CEKPV}
M. Cowling, L. Escauriaza, C.E. Kenig, G. Ponce, and L. Vega, \emph{The Hardy Uncertainty Principle Revisited}. Indiana U. Math. J., \textbf{59} (2010), 2007--2026.

\bibitem{EKPV}
L. Escauriaza, C. E. Kenig, G. Ponce and L. Vega, \emph{Uniqueness properties of solutions to Schr\"odinger equations}, Bull. of Amer. Math. Soc., \textbf{49} (2012), 415--422.

\bibitem{FB}
A. Fern\'andez-Bertolin, \emph{A discrete Hardy's uncertainty principle and discrete evolutions}, \arxiv{1506.00119}


\bibitem{JLMP}
Ph. Jaming, Yu. Lyubarskii, E. Malinnikova, and K.-M. Perfekt. \emph{Uniqueness for discrete Schr\"odinger evolutions}, Rev. Mat. Iberoamericana (to appear). \arxiv{1505.05398}

\bibitem{K}
P. Kuchment. \emph{Quantum graphs: an introduction and a brief survey}, Proc. Sympos. Pure Math., vol 77, p. 291-312, Providence RI, 2008

\bibitem{Levin}
B. Ya. Levin, \emph{Lectures on Entire Functions}, Translations of Mathematical Monographs, Amer. Math. Soc., Providence RI,  1996.

\bibitem{LM}
Yu. I. Lyubarskii, V.A. Marchenko. \emph{Direct and Inverse Multichannel Scattering Problems}, Functional Analysis and its Applications, 2007, vol. 41, no 2, p. 126-142

\bibitem{LM2}
Yu. I. Lyubarskii, V. A. Marchenko. \emph{Inverse problem for small oscillations}, Spectral Analysis and Its Applications, 2007, vol. 41, no2, p.126-142.

\bibitem{MS}
V. A. Marchenko, V. V. Slavin. \emph{Inverse problems in theory of small oscillations} (Russian), Naukova Dumka, Kiev 2015 pp. 218

\bibitem{Tes}
G. Teschl, \emph{Jacobi Operators and Completely Integrable Nonlinear Lattices}, Mathematical Surveys and Monographs, Vol 72, Amer. Math. Soc., Providence RI, 2000.



\end{thebibliography}
\end{document}